\documentclass{amsart}
\usepackage[utf8]{inputenc}

\usepackage{amscd}

\usepackage{amsfonts}

\usepackage{setspace}

\usepackage{version}

\usepackage{amsfonts}
\usepackage{mathrsfs}
\usepackage{amsmath}
\usepackage{amssymb}
\usepackage{amsthm}
\usepackage[maxbibnames=99]{biblatex}
\newtheorem{theorem}{Theorem}

\newtheorem{proposition}{Proposition}
\newtheorem{lemma}{Lemma}

\DeclareMathOperator\supp{supp}

\numberwithin{equation}{section}

\theoremstyle{plain}

\onehalfspace

\title{On the existence of products of primes in arithmetic progressions}
\author{Barnabás Szabó }
\address{Mathematics Institute, Zeeman Building, University of Warwick, Coventry CV4 7AL, England}
\email{\tt szbarna1997@gmail.com}

\begin{document}

\maketitle
\begin{abstract}
    We study the existence of products of primes in arithmetic progressions, building on the work of Ramaré and Walker. One of our main results is that if $q$ is a large modulus, then any invertible residue class mod $q$ contains a product of three primes where each prime is at most $q^{6/5+\epsilon}$. Our arguments use results from a wide range of areas, such as sieve theory or additive combinatorics, and one of our key ingredients, which has not been used in this setting before, is a result by Heath-Brown on character sums over primes from his paper on Linnik's theorem. 

\end{abstract}
\section{Introduction and statements of results}

In this paper we study the existence of products of primes in arithmetic progressions. Let $q$ be a positive integer. The well-known theorem of Dirichlet claims that each invertible residue class $a$ mod $q$ contains infinitely many prime numbers. A much harder problem is to give an upper bound on the size of the smallest such prime. Let $P(a,q)$ be the smallest prime which is congruent to $a$ mod $q$.  In \cite{linnik1944least} and \cite{linnik1944least2} Linnik showed that there are effectively computable absolute constants $C$ and $L$, such that for all $q\geq 1$
\begin{equation}
\label{linnik}
\max_{(a,q)=1} P(a,q)\leq Cq^L .
\end{equation}
It is then an important problem to reduce the value of $L$, which is called Linnik's constant, as much as possible. It was Pan \cite{pan1957least}, who first proved \eqref{linnik} with an explicit $L=10000$  (with some potentially different but still absolute $C$), and since then many authors have succeeded in reducing $L$ (see e.g. \cite{chen1979least}, \cite{jutila1970new}, \cite{graham1981linnik}, \cite{wei1991least}). The current record is held by Xylouris \cite{xylouris2011least}, who showed that one may take $L=5.18$. His proof heavily relies on an earlier paper on the topic by Heath-Brown \cite{heath1992zero}, who showed that one may take $L=5.5$. All the known proofs rely on a meticulous analysis on the distribution of zeros of Dirichlet $L$-functions. The Generalised Riemann Hypothesis (GRH), which states that all the non-trivial zeros of Dirichlet $L$-functions lie on the critical line, implies that one may take $L=2+\epsilon$ for any $\epsilon>0$, and it is further conjectured on probabilistic grounds that $L=1+\epsilon$ can be taken.

A natural extension of the above problem is to show that each invertible residue class mod $q$ contains a product of a small number of small primes. It turns out that elementary methods yield non-trivial results in this direction. In \cite{erdHos1987residues} Erdős, Odlyzko and Sárközy proved various results on products of primes in residue classes conditional upon a strong zero-free region for Dirichlet $L$-functions. In the same paper they conjecture that if $q$ is sufficiently large, then for any $(a,q)=1$ there exist primes $p_1,p_2\leq q$ for which $p_1p_2\equiv a$ $($mod $q)$. This is not known how to prove even if one assumes GRH. We will consider different types of relaxations of this problem. 

Firstly, let us fix some notation. Let $q$ be a modulus which is assumed to be large and let  $G=(\mathbb{Z}/q\mathbb{Z})^*$, so the size of $G$ is $\phi(q)$.
If $k\geq 1$, an $E_k$ number is a positive integer which is a product of exactly $k$ primes. For any $x$ define
\begin{equation}
    E_k(x)=\{n=p_1\cdots p_k : (n,q)=1;\,  p_1,\ldots, p_k\leq x  \},
\end{equation}
which we  regard as a subset of $G$ (so we take elements without multiplicity). Here and throughout the paper $p_i$ always denotes a prime. In this setting, Linnik's theorem states that $E_1(Cq^L)=G$. The aforementioned conjecture can be stated as $E_2(q)=G$. There are two straightforward ways to weaken this problem. To first one is to find $k$ as small as possible such that $E_k(q)=G$. Another way is to show that the size of $E_2(q)$ is relatively large, for example that $|E_2(q)|\geq c\phi(q)$ for some fixed $c>0$. In \cite{walker2016multiplicative} Walker studies both of these problems for prime $q$. In particular Theorem 1 in \cite{walker2016multiplicative} implies $|E_2(q)|\geq (1/64+o(1))\phi(q)$. The idea there is that for any $(a,q)=1$ we have the upper bound
\begin{equation}
\label{walkerupper}
    \# \{(p_1,p_2) : p_1,p_2\leq q\text{ and } p_1p_2\equiv a \ (\text{mod}\; q) \}\leq (64+o(1))\frac{\pi(q)^2}{\phi(q)},
\end{equation}
and then a simple pigeonhole principle argument finishes the proof. The inequality \eqref{walkerupper} is proven by switching to upper bound sieves instead of the counting function of primes. Sieves enjoy nice Fourier analytic properties, so it is much easier to study their distribution in arithmetic progressions. We will improve the above as follows.
\begin{theorem}
\label{t1}
Let $q$ be a cube free number. As $q\to \infty$ we have $|E_2(q)|\geq (3/8+o(1))\phi(q)$.
\end{theorem}
Our improvement comes from two directions. The first one is to notice that to show that $E_2(q)$ is large, it is enough to prove \eqref{walkerupper} for ``almost all'' $a$, so we are free to use an averaging argument. The second ingredient is that we use multiplicative characters instead of additive ones, and prove a strong upper bound on the multiplicative Fourier coefficients of sieves via Burgess's estimates on character sums. This will allow us to work with the multiplicative convolution of a sieve with the counting function of primes instead of working with the convolution of two sieves. We explain how the factor $3/8$ arises in the argument, which seems to be the current limit of our methods. If $0<\alpha<1$ is a real number such that we can show  ``substantial" cancellation in the sum $\sum_{n\leq q^{\alpha+o(1)}} \chi(n)$ for any non-trivial $\chi$ mod $q$, then the factor in Theorem \ref{t1} becomes $\frac{1-\alpha}{2}$.  As it happens Burgess's estimates enable us to take $\alpha=1/4$, and nothing less. Potentially one could improve Theorem \ref{t1} without improving Burgess's estimates by utilising the bilinear structure of sieve weights, however the gain would be small most likely. The factor of $2$ in the denominator represents the loss when taking a sieve instead of the counting function of primes.

Let us now turn to the problem of finding a small $k$ for which $E_k(q)=G$. In this direction Theorem 3 of \cite{walker2016multiplicative} shows that one may take $k=48$. It has been brought to my attention by Aled Walker that he has improved this in his PhD thesis \cite{walker2018topics} by showing $E_{20}(q)=G$, and this seems to be the current best bound. A crucial ingredient, which we will use in our argument as well, is that $E_1(q)$ is not contained inside a proper coset of $G$ when $q$ is a prime (see Lemma 9 of \cite{walker2016multiplicative}).

To prove $E_{20}(q)=G$, Walker first proves the inequality $|E_6(q)|\geq (1/4+o(1))|G|$ and then uses a density increment argument, implemented via Kneser's theorem, similar to classical proof of Schnirelmann's theorem. By using Theorem \ref{t1} we can improve this as follows.
\begin{theorem}
\label{t2}
Let $q$ be a sufficiently large prime. Then $E_6(q)=G$.
\end{theorem}
 Our argument very briefly goes as follows. If $|E_2(q)|+|E_4(q)|>\phi(q)$ then we are done, because for any $g\in G$ the intersection $(g\cdot E_2(q)^{-1}) \cap E_4(q) $ is non-empty by the pigeonhole principle, so $g\in E_2(q)\cdot E_4(q)=E_6(q)$. Hence by Theorem \ref{t1} we may assume $|E_4(q)|\leq (5/8+o(1))\phi(q)$, which implies that $|E_4(q)|\leq 2\cdot |E_2(q)|-\frac{\phi(q)}{8}+o(\phi(q))$. Therefore by Kneser's theorem the stabiliser of $E_4(q)$ is very large, from which we can proceed by case checking. The idea of using Kneser's theorem, which roughly speaking says that the doubling constant of a set should be at least 2 unless it looks like a union of cosets of a large subgroup, comes from \cite{ramare2018products}. We note however that there is still a significant issue we need to handle, for which we use a ``real part trick''. Consider the following scenario. Let $H\leqslant G$ be a subgroup with $G/H\cong C_8=\{1,x,\ldots, x^7\}$, and assume that the primes up to $q$ are contained inside $H\cup xH$. This does not contradict Theorem \ref{t1} as we can have $E_2(q)=H\cup xH\cup x^2H$, moreover $E_6(q)$ is disjoint from $x^7H$, so $E_6(q)\neq G$. To rule out this scenario, we notice that in this case there is a non-trivial character $\chi$ mod $q$ lifted from $G/H$, such that $\chi(p)\in \{ 1, e(1/8)\}$ for any $p\leq q$ (we use the notation $e(z)=e^{2\pi iz}$). This would imply that
\begin{equation}
\label{realparttrick}
    -\Re \frac{L'}{L}(1,\chi)\approx \Re \sum_{p}\frac{\chi(p)\log p}{p}
\end{equation}
is exceptionally large. This contradicts a result of Heath-Brown on Linnik's theorem, namely Lemma 5.2 of \cite{heath1992zero}. It seems possible that one can improve Theorem \ref{t2} by following a similar line of thoughts. By Theorem \ref{t1} it should be that $E_4(q)=E_2(q)\cdot E_2(q)=G$, unless $E_2(q)$ has very special properties, but now Kneser's theorem alone does not suffice. One can try to make use of some additional arguments, for example it can be proven that $\cup_{k=1}^4 E_k(q)=G$, which could be helpful in ruling out case that $E_2(q)$ is contained inside a geometric progression of size $\phi(q)/2$ (if this was true then certainly $E_4(q)\neq G$).

Let us now turn to our final result. The problem of finding $E_3$ numbers in all residue classes was studied by Walker and Ramaré in \cite{ramare2018products}. They showed that for any $q\geq 2$, one has $E_3(q^{16/3})=G$. This was further improved in \cite{ramare2020product} by Ramaré, Srivastav and Serra, where it was shown that for any $q\geq 2$ one has $E_3((650q)^3)=G$. We note that these results are fully explicit and have elementary proofs. In their forthcoming work, Balasubramanian, Ramaré and Srivastav prove that for any fixed $\epsilon>0$, if $q$ is large enough, then $E_3(q^{3/2+\epsilon} )=G$, moreover if $q$ is cube-free one has $E_3(q^{4/3+\epsilon} )=G$. This work came to the author's knowledge via private communication with Olivier Ramaré. Using a sieve result referenced in their paper, we will improve this as follows.
\begin{theorem}
\label{t3}
Let $\epsilon>0$. There exists a $q_0(\epsilon)$, such that for any $q\geq q_0(\epsilon)$ one has $E_3(q^{6/5+\epsilon} )=G$.
\end{theorem}
Unlike the previous theorems, here there is no restriction on $q$. Note that $3\cdot 6/5<5$, so our bound for the smallest $E_3$ number in an arithmetic progression is better than the current best bound on the smallest prime in an arithmetic progression and the proof is much simpler. Of course, it is expected that finding $E_3$ numbers should be much easier than finding primes. The proof uses an average version of the Brun-Titchmarsh theorem by Iwaniec \cite{iwaniec1982brun} (together with comments by Mikawa in \cite{mikawa1991brun}), which implies that $|E_1(q^{6/5+\epsilon} )|>\frac{11}{32}\phi(q)$ for large $q$. This result came to the author's knowledge when reading the manuscript mentioned above by Balasubramanian, Ramaré and Srivastav, and the author is very grateful to Olivier Ramaré for sending him their work. By Kneser's theorem and a sieve result by Mikawa in \cite{mikawa1989almost}, we deduce that if the statement is false, then the primes up to $q^{6/5}$ must be contained inside around one-third proportion of the cosets of a large subgroup $H\leqslant G$. This would, roughly speaking, imply that there is some non-trivial character $\chi$ mod $q$ for which $\chi(p)$ has large real part for each $p\leq q^{6/5}$, which implies that
\eqref{realparttrick}
is exceptionally large. This is the same trick as the one we apply in Theorem \ref{t2}. As an example, the simplest possibility we have to rule out is that there exists $H\leqslant G$ with $G/H=\{1,x,\ldots x^4 \}$, such that the primes up to $q^{6/5}$ are contained in $x^2H\cup x^3H$.

We end this section by mentioning some related results. We can view this topic as a modular hyperbola problem, where the variables are restricted to be primes. In \cite{shparlinski2012modular} Shparlinski gives an elaborate survey of such problems and in \cite{shparlinski2019short} proves related results on $P_k$ numbers, where a $P_k$ number is a product of at most $k$ many primes. One can substantially improve Theorem \ref{t3} for ``almost all" $q$. As an example, in \cite{klurman2019multiplicative} the authors prove the existence of an $E_3$ number less than $q^{2+\epsilon}$ in all residue classes mod $q$, provided that $q$ is sufficiently smooth, or that $q$ is a non-exceptional prime (see Theorem 2.1 and 2.2 in the referenced paper). 
\section{Some Fourier analysis and notation}
In this section we recall some of the basic facts from discrete Fourier analysis and fix some notation. Let $q$ be a large modulus and let $G=(\mathbb{Z}/q\mathbb{Z})^*$. For any $z\geq 1$ let $f_z:G\rightarrow \{0,1\}$ be defined as follows. Let $n\in G$ where now $n$ is thought of as an integer with $0\leq n\leq q-1$. Let $f_z(n)=1$ if $n\geq z$ and $n$ is a prime, let $f_z(n)=0$ otherwise.

For any $g:G\rightarrow \mathbb{C}$ and a multiplicative character $\chi:G\rightarrow \mathbb{C}$ we define the multiplicative Fourier transform as
\begin{equation*}
    \hat{g}(\chi)=\sum_{n\in G}g(n)\bar\chi(n).
\end{equation*}
Note that for the trivial character $\chi_0$ we have $\hat{g}(\chi_0)=\sum_{n\in G} g(n)$. For functions $g,h:G\rightarrow \mathbb{C}$ we define their convolution $g*h:G\rightarrow \mathbb{C}$, where for any $n\in G$ we have
\begin{equation*}
    g*h(n)=\sum_{ab=n}g(a)h(b).
\end{equation*}
The convolution operator has the nice property that $\widehat{g*h}(\chi)=\hat{g}(\chi)\hat{h}(\chi)$. We also have Parseval's identity, namely
\begin{equation*}
    \sum_{n\in G}|g(n)|^2=\frac{1}{\phi(q)}\sum_{\chi}|\hat{g}(\chi)|^2.
\end{equation*}
The Fourier inversion formula states that
\begin{equation*}
    g(n)=\frac{1}{\phi(q)}\sum_{\chi}\hat{g}(\chi)\chi(n).
\end{equation*}
Here and throughout the paper $\chi$ is always summed over the full set of Dirichlet characters mod $q$ unless otherwise stated.
When dealing with error terms $o(1)$ should be translated as ``some quantity which tends to 0 as $q\to \infty$''. We use Landau's big $O$ notation and Vinogradov's $\ll$ notation. $p$ and $p_i$ will always denote a prime.

\section{Results used in the paper}
We now state several well-known results that we will need later on. It is not needed to know the proofs of these to understand later parts of the paper. We start with the existence of appropriate sieve weights, which we need for the proof of Theorem \ref{t1}. This proposition is essentially Proposition 4 in \cite{walker2016multiplicative}.
\begin{proposition}
Let $\delta>0$ be fixed. For sufficiently large $q$, if $D\leq q^{1-\delta}$, there exists a function $w:G\rightarrow \mathbb{R}_{\geq 0}$ with the following properties. Let $z=D^{1/2}$.
\begin{itemize}
    \item (Upper bound property) For all $n\in G$ we have $w(n)\geq f_z(n)$.
    
    \item (Well approximation) We have $\hat{w}(\chi_0)\leq (1+o(1))\frac{2\log q}{\log D}\hat{f}_z(\chi_0)$.
    
    \item (Sieve weight structure) For some sequence $(\lambda_d)_{1\leq d<D}$ such that $|\lambda_d|\leq q^{o(1)}$ the following holds. Let $n\in G$ be represented as an integer $1\leq n\leq q-1$. Then we have  $w(n)=\sum_{d|n}\lambda_d$.
\end{itemize}
\end{proposition}
\begin{proof}
The standard Selberg sieve weights suffice. We define the weights $\lambda_d$ for explicitness. Let $z=D^{1/2}$ and let
\begin{equation*}
    G(z)=\sum_{n\leq z}\frac{\mu^2(n)}{\phi(n)},
\end{equation*}
and for any $1\leq d\leq z$ let
\begin{equation*}
    \rho_d=\frac{d\mu(d)}{G(z)}\sum_{\substack{n\leq z\\ d|n}}\frac{\mu^2(n)}{\phi(n)}.
\end{equation*}
If $d>z$ then let $\rho_d=0$. Define 
\begin{equation*}
    \lambda_d=\sum_{d_1,d_2: [d_1,d_2]=d} \rho_{d_1}\rho_{d_2}.
\end{equation*}
For any $1\leq n\leq q-1$ we define 
\begin{equation*}
    w(n)=\sum_{d|n}\lambda_d
\end{equation*}
and make this a function on $G$ in the obvious way. For a proof that these weights satisfy the above conditions and an elaborate introduction to the Selberg sieve we refer the reader to Chapter 4 and 5 of \cite{halberstam2013sieve}.
\end{proof}

Next, we need Burgess's estimates on character sums for the proof of Theorem \ref{t1}.
\begin{proposition}
Let $q$ be a cube free number and $\chi$ a non-trivial character mod $q$. For any $1\leq N\leq q$ and $r\geq 1$ integer we have
\begin{equation*}
    \bigg|\sum_{n\leq N}\chi(n)\bigg|\ll_r N^{1-1/r}q^{\frac{r+1}{4r^2}+o(1)}.
\end{equation*}
\end{proposition}
\begin{proof}
For the original treatment of character sums by Burgess see \cite{burgess1957distribution}, \cite{burgess1962character}, \cite{burgess1962characterlseries}, \cite{burgess1963character}. It is Theorem 2 in \cite{burgess1963character} which is the statement of our proposition. When $q$ is prime, see Theorem 9.27 in \cite{montgomery2007multiplicative} for a simpler argument.
\end{proof}
For the proof of Theorem 2 and 3 we need Kneser's theorem for finite abelian groups, which we state in multiplicative notation.
\begin{proposition}
Let $(G,\cdot)$ be a finite abelian group and $A\subset G$. Let $H$ be the stabiliser of $A\cdot A=\{a_1a_2: a_1,a_2\in A\}$, that is $H=\{g\in G: \, g\cdot A \cdot A=A\cdot A\}$. Then we have
$$|A\cdot A|\geq |A\cdot H|+|A\cdot H|-|H|.$$
\end{proposition}
\begin{proof}
This is a special case Theorem 5.5 in \cite{tao2006additive}, where we take $A=B$. Another treatment of this theorem can be found in in Chapter 4 of \cite{nathanson1996additive}. Kneser's original work on the topic is in German and can be found in \cite{kasch1955abschatzung} and \cite{kneser1956summenmengen}.
\end{proof}
Next, we use a result of Iwaniec on the Brun-Titchmarsh theorem for the proof of Theorem \ref{t3}. The strongest form of the Brun-Titchmarsh theorem is not sharp enough for us, however Iwaniec was able to obtain a stronger upper bound for the number of primes in arithmetic progressions for ``almost-all'' residue classes. This in turn implies a corresponding lower bound on the number of residue classes occupied by primes.
\begin{proposition}
Let $\epsilon>0$ be fixed. If $q$ is large then one has
$$|E_1(q^{6/5+\epsilon})|>\frac{11}{32}\phi(q).$$
\end{proposition}
\begin{proof}
As noted in the Remark of \cite{mikawa1991brun}, this can be proved using the techniques in Section 2 of \cite{iwaniec1982brun}. The remark made by Mikawa implies that for every fixed $\epsilon_1,\epsilon_2,\epsilon_3>0$, if $q$ is large enough and $\alpha\geq 6/5+\epsilon_1$, then for all but at most $\epsilon_2 \phi(q)$ many invertible residue classes $a$, we have
\begin{equation*}
\label{iwaniec}
    \pi(q^{\alpha},q,a)\leq \frac{2+\epsilon_3}{1-\frac{3}{8\alpha}} \frac{q^{\alpha} }{\phi(q)\log q^{\alpha}}.
\end{equation*}
 For brevity, let us call
$$X:=\frac{q^{\alpha} }{\phi(q)\log q^{\alpha}}.$$
For any ``exceptional'' $a$, we apply a crude version of the Brun-Titchmarsh inequality, namely if $\alpha\geq 6/5$, then for any $a$  we have
$$\pi(q^{\alpha},q,a)\leq 100X.$$
By the prime number theorem we infer
\begin{equation*}
   (1+o(1))X\phi(q)=\sum_{a\in E_1(q^{\alpha}) } \pi(q^{\alpha}, q,a)\leq 100X\epsilon_2\phi(q)+\frac{2+\epsilon_3}{1-\frac{3}{8\alpha}} X |E_1(q^{\alpha})|,
\end{equation*}
therefore
\begin{equation}
 \label{beeg}
    |E_1(q^{\alpha})|\geq \phi(q) \frac{(1+o(1)-100\epsilon_2 )(1-\frac{3}{8\alpha})}{2+\epsilon_3}.
\end{equation}
Now let us choose $\epsilon_1=\epsilon$ and $\alpha=6/5+\epsilon$. We then have
$$\frac{1-\frac{3}{8\alpha}}{2}=\frac{11+40\epsilon/3}{32+80\epsilon/3}>\frac{11}{32}$$
so if we choose $\epsilon_2$ and $\epsilon_3$ small enough, then \eqref{beeg} implies the proposition.

\end{proof}

We will also need a sieve result by Mikawa for the proof of Theorem \ref{t3}, which shows the existence of almost primes in almost all arithmetic progressions.

\begin{proposition}
If $q$ is sufficiently large then
\begin{equation*}
    |E_1(q^{6/5})\cup E_2(q^{6/5})|= (1+o(1))\phi(q).
\end{equation*}
\end{proposition}
\begin{proof}
This follows from Theorem 1 in \cite{mikawa1989almost}.
\end{proof}
We finally state a special case of a result by Heath-Brown, which says that a weighted character sum over primes cannot have very large real part.
\begin{proposition}
\label{propheath}
 Let $\alpha>0$ be a real number. Let $f:\mathbb{R}_{\geq 0} \rightarrow \mathbb{R}$ be defined as
\begin{equation*}
f(t)=
\begin{cases}
\alpha-t & \text{ if } 0\leq t\leq \alpha, \\ 
0 & \text{ if }  t>\alpha. \\
\end{cases}
\end{equation*}
For the trivial character we have
\begin{equation*}
     \sum_{p}\frac{\chi_0(p)\log p}{p}f\bigg(\frac{\log p}{\log q}\bigg)= (\alpha^2/2+o(1))\log q.
\end{equation*}
If $\chi$ is a non-trivial character mod $q$ of bounded order, then
\begin{equation*}
    \Re \sum_{p}\frac{\chi(p)\log p}{p}f\bigg(\frac{\log p}{\log q}\bigg)\leq (\alpha/8+o(1))\log q.
\end{equation*}
\end{proposition}
\begin{proof}
The first estimate is essentially Mertens' first theorem (or see Lemma 5.3 in \cite{heath1992zero}), whereas the second estimate follows from Lemma 5.2 of \cite{heath1992zero}, which we now explain why. Most of the discussion is straightforward, however there is a vital point at the end we need to address. When applying the lemma, we take $s=1$; $f$ and $\chi$ as in our proposition. We can also take $\phi=\frac{1}{4}$ in the lemma because $\chi$ has bounded order (the quantity $\phi$ is introduced in Lemma 2.5 there). Note that $\mathscr{L}=\log q$ in the paper. Our choice of $f$ satisfies the conditions of the lemma (see the definition of Condition $1$ at the beginning of Section 5). The contribution from prime powers in the left hand side is negligible as
\begin{equation*}
    \sum_{\substack{p^k \\ k\geq 2 }}\frac{\chi(p^k)\log p }{p^k}f\bigg(\frac{\log p^k }{\log q}\bigg)\ll \sum_{\substack{p^k \\ k\geq 2 }}\frac{\log p}{p^k }\ll 1 .
\end{equation*}
Now the crucial part of the adaptation of Lemma 5.2 in \cite{heath1992zero} is that the Laplace transform of $f$, namely
\begin{equation*}
F(z)=\int_0^{\infty}f(t)e^{-zt}dt=
\begin{cases}
\frac{e^{-\alpha z}-(1-\alpha z)}{z^2} & \text{ if } z\neq 0, \\ 
\frac{\alpha^2}{2} & \text{ if }  z=0 \\
\end{cases}
\end{equation*}
satisfies $\Re F(z)\geq 0$ when $\Re z\geq 0$, so the contribution from the zeros of $L(s,\chi)$ on the right hand side is non-positive, thus can be ignored. This is because we may write $f$ as a convolution, $f=g*g$, where $g:\mathbb{R}\rightarrow \mathbb{R}$ with
\begin{equation*}
g(t)=
\begin{cases}
1 & \text{ if } -\alpha/2 \leq t\leq \alpha/2, \\ 
0 & \text{ otherwise, }   \\
\end{cases}
\end{equation*}
and the rest follows from the discussion at the beginning of Section 7 in \cite{heath1992zero}.
\end{proof}

\section{Proof of Theorem 1}
We start with a lemma that gives a lower bound on the support of a function on $G=(\mathbb{Z}/q\mathbb{Z})^*$.
\begin{lemma}
\label{l1}
Let $g,h:G\rightarrow \mathbb{R}_{\geq 0}$ and assume that for all $n\in G$ we have $0\leq g(n)\leq h(n)$. Let $\supp(g):= \{n\in G: g(n)>0\}$. Assume $h$ is not identically $0$. For any $\epsilon>0$ we have
\begin{equation*}
\label{lemma1}
    |\supp(g)|\geq \min \bigg( (1-\epsilon)|G|\cdot  \frac{\hat{g}(\chi_0)  }{\hat{h}(\chi_0)},\; |G|\cdot \frac{\epsilon^2\cdot \hat{g}(\chi_0)^2 }{\sum_{\chi\neq \chi_0 } |\hat{h}(\chi)|^2 }\bigg).
\end{equation*}
Here the second term is taken to be $\infty$ if the denominator is $0$.
\end{lemma}
\begin{proof}

By the Cauchy-Schwarz inequality
\begin{equation}
\label{am-gm}
    |\supp(g)| \sum_{n\in \supp(g) }\bigg(h(n)-\frac{\hat{h}(\chi_0)}{|G|}\bigg)^2\geq \bigg( \sum_{n\in \supp(g)} h(n)-|\supp(g)|\cdot \frac{\hat{h}(\chi_0) }{|G|}\bigg)^2. 
\end{equation}
By the assumption $g(n)\leq h(n)$ we have
\begin{equation*}
    \sum_{n\in \supp(g)} h(n)\geq \sum_{n\in \supp(g)} g(n)= \sum_{n\in G}g(n)=\hat{g}(\chi_0).
\end{equation*}
If $|\supp(g)|\geq   (1-\epsilon)|G|\cdot  \frac{\hat{g}(\chi_0)  }{\hat{h}(\chi_0)}$, then we are done. Otherwise $|\supp(g)|\cdot \frac{\hat{h}(\chi_0) }{|G|}\leq (1-\epsilon)\hat{g}(\chi_0)$, so the right hand side of \eqref{am-gm} is at least $\epsilon^2 \hat{g}(\chi_0)^2$. 
As for the left hand side of \eqref{am-gm} we have by Parseval's identity

\begin{equation*}
    \sum_{n\in \supp(g) }\bigg(h(n)-\frac{\hat{h}(\chi_0)}{|G|}\bigg)^2\leq \sum_{n\in G }\bigg(h(n)-\frac{\hat{h}(\chi_0)}{|G|}\bigg)^2=\frac{1}{|G|}\sum_{\chi\neq \chi_0} |\hat{h}(\chi_0)|^2.
\end{equation*}
Thus \eqref{am-gm} implies
\begin{equation*}
    \frac{|\supp(g)|}{|G|}\sum_{\chi\neq \chi_0} |\hat{h}(\chi_0)|^2\geq \epsilon^2 \hat{g}(\chi_0)^2,
\end{equation*}
from which the lemma follows if $\sum_{\chi\neq \chi_0 } |\hat{h}(\chi)|^2 \neq 0$, and if $\sum_{\chi\neq \chi_0 } |\hat{h}(\chi)|^2=0$, then $\hat{g}(\chi_0)=0$ from which the lemma follows trivially.


\end{proof}
The next lemma states that sieve functions have small non-trivial Fourier coefficients.
\begin{lemma}
\label{l2}
Let $\epsilon>0$ be given. There exists a $\delta=\delta(\epsilon)$ such that the following holds. Let $q$ be a cube-free number and $w:G\rightarrow \mathbb{R}_{\geq 0}$ be the sieve function provided by Proposition 1 with level $D=q^{3/4-\epsilon}$. Then, if $\chi$ is a non-trivial character mod $q$ we have
\begin{equation*}
    |\hat{w}(\chi)|\leq q^{1-\delta+o(1)}.
\end{equation*}
\end{lemma}
\begin{proof}
By the sieve structure property we have
\begin{equation*}
    |\hat{w}(\chi)|=\bigg|\sum_{n\in G}w(n)\chi(n)\bigg|=\bigg| \sum_{d<D}\lambda_d \sum_{\substack{n: n<q\\ d|n}}\chi(n)\bigg|\leq q^{o(1)}\sum_{d<D} \bigg| \sum_{n\leq q/d}\chi(n)\bigg|.
\end{equation*}
Now we use Proposition 2 (Burgess's estimates on character sums) to bound the expressions inside the absolute value. Let $\alpha_0=1$ and for any $r\geq 1$ let $\alpha_r=\frac{r^2+3r+1}{4r(r+1)}$. Note that $\alpha_0\geq \alpha_1\geq \ldots$ and $\lim_{r\to \infty} \alpha_r=\frac{1}{4}$. It turns out, when $q/d$ is in the range $[q^{\alpha_r},q^{\alpha_{r-1}}]$ we apply Proposition 2 with the value $r$. For any $r\geq 1$ we have

\begin{equation*}
\begin{split}
    \sum_{q^{1-\alpha_{r-1}}\leq d\leq q^{1-\alpha_r}}\bigg| \sum_{n\leq q/d}\chi(n)\bigg| 
    \ll_r & \sum_{q^{1-\alpha_{r-1}}\leq d\leq q^{1-\alpha_r}} \Big(\frac{q}{d}\Big)^{1-1/r} q^{\frac{r+1}{4r^2}+o(1)} \\
    = &\; q^{1-\frac{1}{r}+\frac{r+1}{4r^2}+o(1)}\sum_{q^{1-\alpha_{r-1}}\leq d\leq q^{1-\alpha_r}} \Big(\frac{1}{d}\Big)^{1-1/r} \\
    \ll_r & \; q^{1-\frac{1}{r}+\frac{r+1}{4r^2}+o(1)}\cdot q^{(1-\alpha_r)/r} \\
    = & \; q^{1-\frac{1}{4r(r+1)}+o(1)}. \\
\end{split}
\end{equation*}
Let $R=R(\epsilon)$ be such that $\alpha_R\leq 1/4+\epsilon$. As $\epsilon$ is fixed, any expression in terms of $\epsilon$ and $R$ is $q^{o(1)}$. We therefore have
\begin{equation*}
     |\hat{w}(\chi)|\leq q^{o(1)}\sum_{1\leq r\leq R} \sum_{q^{1-\alpha_{r-1}}\leq d\leq q^{1-\alpha_r}}\bigg| \sum_{n\leq q/d}\chi(n)\bigg| \leq q^{o(1)} \sum_{1\leq r\leq R}q^{1-\frac{1}{4r(r+1) } }\leq q^{1-\frac{1}{4R(R+1)}+o(1) },
\end{equation*}
so the lemma holds with $\delta= \frac{1}{4R(R+1)}$.
\end{proof}
We can now easily deduce Theorem \ref{t1}.
Let $w:G\rightarrow R_{\geq 0}$ be the sieve function provided by Proposition 1 with level $D=q^{3/4-\epsilon}$, where $\epsilon>0$ is a small fixed number. Let $f=f_z$ with $z=D^{1/2}$, so $f\leq w$. We apply Lemma \ref{l1} with $g=f*f$, $h=f*w$ and the same $\epsilon$. Our lower bound for $\supp(f*f)$ will come from the first expression as we will show that the second expression is larger than $q$.

By Parseval's identity $\sum_{\chi} |\hat{f}(\chi)|^2=\phi(q)\sum_{n\in G} |f(n)|^2\leq q^2$. By Lemma \ref{l2} there is some $\delta>0$ depending only on $\epsilon$ such that
\begin{equation*}
\sum_{\chi\neq \chi_0}|\hat{h}(\chi)|^2=\sum_{\chi\neq \chi_0}|\hat{f}(\chi)|^2|\hat{w}(\chi)|^2\leq \max_{\chi\neq \chi_0}|\hat{w}(\chi)|^2\cdot  \sum_{\chi} |\hat{f}(\chi)|^2\leq q^{4-\delta+o(1)}.
\end{equation*}
We also have  $|\hat{g}(\chi_0)|^2=|\hat{f}(\chi_0)|^4\geq q^{4-o(1)}$ hence if $\epsilon$ stays fixed, then for large $q$
\begin{equation*}
    \frac{\epsilon^2\cdot  |\hat{g}(\chi_0)|^2}{\sum_{\chi\neq \chi_0 } |\hat{h}(\chi)|^2}>q
\end{equation*}
indeed. The first expression in Lemma \ref{l1} is easy to calculate by Proposition 1,
\begin{equation*}
    \frac{\hat{g}(\chi_0)}{\hat{h}(\chi_0)}=\frac{\hat{f}(\chi_0)}{\hat{w}(\chi_0)}\geq \frac{3}{8}-\frac{\epsilon}{2}+o(1).
\end{equation*}
Therefore if $q$ is large then
\begin{equation*}
    |E_2(q)|\geq |\supp(f*f)|\geq (1-\epsilon)\Big(\frac{3}{8}-\frac{\epsilon}{2}+o(1)\Big)\phi(q)\geq \Big(\frac{3}{8}-\epsilon\Big)\phi(q).
\end{equation*}
As $\epsilon$ was arbitrary Theorem \ref{t1} follows.

\section{Proof of Theorem \ref{t2}}
For simplicity, let us denote $A:=E_1(q)\subset G$, so our goal is to show that $A^{(6)}=G$, where $A^{(k)}$ denotes the $k$-fold product of $A$ with itself. If $q$ is sufficiently large then by Theorem \ref{t1} 
\begin{equation*}
    |A^{(2)}|\geq \Big(\frac{3}{8}-10^{-4} \Big)\phi(q).
\end{equation*}
Let $H\leqslant G$ be the stabiliser of $A^{(4)}$. Then by Proposition 3 (Kneser's theorem)
\begin{equation*}
    |A^{(4)}|\geq 2\cdot |A^{(2)}|-|H|.
\end{equation*}
If $H$ has index at least $9$, then $|H|\leq \phi(q)/9$, so
\begin{equation*}
    |A^{(4)}|> \Big(\frac{5}{8}+10^{-4}  \Big)\phi(q),
\end{equation*}
so $|A^{(4)}|+|A^{(2)}|>|G|$, which implies $A^{(6)}=G$. So we may assume that $G_0=G/H$ is a group of order at most 8. Let $\pi:G\rightarrow G_0$ be the quotient map and let $\pi(A)=A_0$. Notice that $H$ lies inside the stabiliser of $A^{(6)}$, so $A^{(6)}=A^{(6)}\cdot H=(A\cdot H)^{(6)}$, therefore it is enough to prove $A_0^{(6)}=G_0$. By Lemma 9 of \cite{walker2016multiplicative}, if $q$ is a prime $A$ is not contained insider a proper coset of $G$ (i.e. a coset of a proper subgroup), so the same is true for $A_0$ inside $G_0$ (this is the only part where we need $q$ to be a prime). In particular $|A_0|\geq 2$. By Proposition 2.2 in \cite{tao2006additive}, if $A_0^{(k)}\neq G_0$, then $|A_0^{(k+1)}|> |A_0^{(k)}| $, so we are done if $|A_0|\geq 3$ or $|G_0|<8$. So we assume that $|A_0|=2$ and $|G_0|=8$. Then $G_0$ cannot be isomorphic to $C_2\times C_2\times C_2 $, or $C_2\times C_4$, since any two element subset $A_0=\{a,b\}$ is contained inside a coset of the proper subgroup generated by $ab^{-1}$. 

Therefore we must have $G_0\cong C_8=\{1,x, \ldots ,x^7\}$. Let $A_0=\{x^i,x^j\}$, so if $1\leq p< q$ is regarded as an element of $G$ then $p\in x^iH\cup x^jH$ (here and throughout we abuse notation and regard $x$ as an element of $G$ by taking one of its preimages under the quotient map). Define $g:\mathbb{R}_{\geq 0}\rightarrow \mathbb{R}$ by
$$g(x)=
\begin{cases}
1-x & \text{ if } 0\leq x\leq 1, \\ 
0 & \text{ if }  x>1. \\
\end{cases}$$

Let 
\begin{equation*}
    \begin{split}
        M:=\frac{2}{\log q}\sum_{\substack{p<q\\ p\in x^iH }} \frac{\log p}{p}g\Big(\frac{\log p}{\log q}\Big), \\
        N:=\frac{2}{\log q}\sum_{\substack{p<q\\ p\in x^jH }} \frac{\log p}{p}g\Big(\frac{\log p}{\log q}\Big). \\
    \end{split}
\end{equation*}
We will obtain a contradiction by proving bounds on $M$ and $N$ which are inconsistent with each other. We clearly have $M,N\geq 0$. By the first part of Proposition \ref{propheath} we have
\begin{equation}
\label{sum}
    M+N=1+o(1).
\end{equation}
By the second part of Proposition \ref{propheath}, if $\chi\neq \chi_0$ is of bounded order, we have
\begin{equation}
\label{realpart}
   \Re\big( \chi(x^i)M+\chi(x^j)N\big)\leq 1/4+o(1).
\end{equation}
We may assume that $i$ is even and $j$ is odd, otherwise $A_0$ is contained inside a coset of the proper subgroup generated by $x^{i-j}$. We do a quick case-checking with respect to the value of $i$. Define $\chi_1$ mod $q$ to be the character lifted from a character on $G_0$ (i.e. $\chi(H)=\{1\}$), such that $\chi_1(x)= e(1/8)$, thus any power of $\chi_1$ is of bounded order.

If $i=0$, in \eqref{realpart} choose $\chi=\chi_1^{\bar{j}}$, where $\bar{j}$ is the multiplicative inverse of $j$ mod $8$. Then $\chi(x^i)=1$ and $\chi(x^j)=\chi_1(x^{j\bar{j}})=e(1/8)$, so \eqref{realpart} becomes
\begin{equation*}
    M+N/\sqrt{2}\leq 1/4+o(1),
\end{equation*}
but this contradicts \eqref{sum}, because for large $q$
\begin{equation*}
    M+N/\sqrt{2}\geq (M+N)/\sqrt{2} >1/2.
\end{equation*}

We next consider the case of $i=2$ or $6$, in which case in \eqref{realpart} choose $\chi=\chi_1^{\bar{j}}$. Note that $\chi(x^i)=\chi_1(x^{i\bar{j}})=e\big(\frac{i\bar{j}}{8} \big)$, which is either $e(2/8)$ or $e(6/8)$ so its real part is $0$ either way. On the other hand $\chi(x^{j})=\chi_1(x^{j\bar{j}})=e(1/8)$. So we get
\begin{equation*}
    \frac{N }{\sqrt{2}}\leq 1/4+o(1).
\end{equation*}
Next, we substitute $\chi=\chi_1^4$, So $\chi(x^i)=\chi_1(x^{4i})=1$ and $\chi(x^j)=\chi_1(x^{4j})=-1$, so \eqref{realpart} becomes
\begin{equation*}
    M-N\leq 1/4+o(1).
\end{equation*}
The last two inequalities together yield
\begin{equation*}
    M+N=(M-N)+\sqrt{8}\cdot  \frac{N}{\sqrt{2}}\leq \frac{\sqrt{8}+1}{4}+o(1),
\end{equation*}
which contradicts \eqref{sum} for large $q$.

Finally, if $i=4$, in \eqref{realpart} choose $\chi=\chi_1^2$, which yields
\begin{equation*}
    M\leq 1/4+o(1).
\end{equation*}
Next, choose $\chi=\chi_1^{\bar{j}}$ to get
\begin{equation*}
    -M+N/\sqrt{2}\leq 1/4+o(1).
\end{equation*}
These two inequalities together imply
\begin{equation*}
    M+N=\sqrt{2} \cdot (-M+N/\sqrt{2})+ (\sqrt{2}+1)M\leq \frac{\sqrt{8}+1}{ 4}+o(1),
\end{equation*}
which contradicts \eqref{realpart} for large $q$.

We have proved that $|A_0|=2$ and $|G_0|=8$ is not possible so Theorem \ref{t2} is proved.

\section{Proof of theorem \ref{t3}}
Let $\epsilon>0$ be fixed, let $A=E_1(q^{6/5+\epsilon})$. We will show that $A\cdot A\cdot A=G$ if $q$ is large enough. By Proposition 4, we have $|A|>\frac{11}{32}\phi(q)$. Let $H$ be the stabiliser of $A\cdot A$. Let $\pi: G\rightarrow G/H$ be the projection map, $\pi(A)=B$ so $\pi(A\cdot A)=B\cdot B$. Our next goal is the following lemma. 
\begin{lemma}
\label{casecheck}
    Either Theorem \ref{t3} holds, or there is some $0\leq k\leq 9$, such that $|G/H|=3k+2$, $|B|=k+1$, moreover $B$ and $B\cdot B$ are complement sets in $G/H$.
\end{lemma}
\begin{proof}
By Proposition 3, if $A$ meets $\lambda$ many cosets of $H$, i.e. $\lambda=|B|$, then
\begin{equation}
    \label{kneserlambda}
    |A\cdot A|\geq (2\lambda-1)|H|.
\end{equation}
Let us write $|G/H|=3k+r$ for some $r\in \{0,1,2\}$ and $k\geq 0$. In each case we have
\begin{equation}
\label{six}
    |A|>\frac{|G|}{3}=\frac{|G/H|}{3} \cdot \frac{|G|}{|G/H|}\geq k|H|,
\end{equation}
 so $\lambda\geq k+1$. If $r=0$ or $1$ then by \eqref{kneserlambda} and \eqref{six} we get
\begin{equation*}
\label{anyad}
    |A|+|A\cdot A|>(k+2\lambda-1)|H|\geq (3k+1)|H|\geq |G|,
\end{equation*}
which implies $A\cdot A\cdot A=G$, so Theorem \ref{t3} holds. So we may assume that $|G/H|=3k+2$ indeed.

By \eqref{six} we have seen $\lambda \geq k+1$. If $\lambda\geq k+2$, then by \eqref{kneserlambda} and \eqref{six} we have
\begin{equation*}
    |A|+|A\cdot A|>(k+2\lambda-1)|H|>(3k+2)|H|=|G|,
\end{equation*}
which gives us the theorem. So we may assume $\lambda=k+1$, i.e. $|B|=k+1$ indeed.

If $|B\cdot B|\geq 2k+2$, then $|A\cdot A|\geq (2k+2)|H|$, because $A\cdot A$ is a union of cosets of $H$, and $\pi(A\cdot A)=B\cdot B$. Using \eqref{six} we get
\begin{equation*}
    |A|+|A\cdot A|> (3k+2)|H|=|G|,
\end{equation*}
which implies the theorem.  So we can assume $|B\cdot B|\leq 2k+1$, but by \eqref{kneserlambda} we have $|B\cdot B|\geq 2k+1$, so in fact $|B\cdot B|=2k+1$. 

Notice that
\begin{equation*}
    \frac{11}{32}|G|< |A|\leq (k+1)|H|=\frac{k+1}{3k+2} \cdot|G|,
\end{equation*}
from which we deduce $k<10$ indeed.

We have $|B|+|B\cdot B|=3k+2=|G/H|$, therefore if $B$ and $B\cdot B$ are not complements in $G/H$, then $|B\cup B\cdot B|\leq 3k+1$, which would imply that $|A\cup A\cdot A|\leq (3k+1)|H|=\frac{3k+1}{3k+2}\cdot |G|$. This would contradict Proposition 5, namely that $|A\cup A\cdot A|=(1+o(1))|G|$. So we may assume that $B$ and $B\cdot B$ are complements indeed, and the lemma is proved.
\end{proof}
By the lemma, henceforth we assume that for some $0\leq k\leq 9$ we have $|G/H|=3k+2$, $|B|=k+1$, moreover $B$ and $B\cdot B$ are complement sets in $G/H$. Note also that $B\cdot B$ has trivial stabiliser in $G/H$, otherwise $A\cdot A$ would have a stabiliser larger than $H$.

We quickly rule out the case of $k=0$. If $k=0$, then $H$ is an order $2$ subgroup of $G$, and $A$ is contained inside a coset of $H$. Hence there is a quadratic character $\chi$ mod $q$ which is constant on $A$. We cannot have $\chi(A)=\{1\}$, because then for any $1\leq n\leq q$ and $(n,q)=1$ we would have $\chi(n)=1$ by the multiplicative property of $\chi$, but $\chi$ is not the trivial character. We cannot have $\chi(A)=\{-1\}$ by Theorem 1.3 of \cite{pollack2017bounds} (this is a strengthened and generalised version of Vinogradov's theorem on the least prime quadratic residue).

Let us call $G_0=G/H$. Our next lemma shows that $G_0$ must be cyclic and $B$ is essentially determined.

\begin{lemma}
\label{taovu}
Let $(G_0,\cdot)$ be an abelian group of order $3k+2$ for some $k\geq 1$, and let $B\subset G_0$ with $|B|=k+1$. Assume that $B\cdot B$ is the complement set of $B$ inside $G_0$ and also that $B\cdot B$ has trivial stabiliser.  Then $G_0\cong C_{3k+2}$ and for some $x\in G_0$ which generates $G_0$ we have $B= \{x^{k+1},\ldots,x^{2k+1 }\}$.
\end{lemma}
\begin{proof}
Note that $|B|>1$ and $|B\cdot B|<2|B|$. Therefore by exercise 5.1.11 in \cite{tao2006additive}, either $B\cdot B$ is a geometric progression, or there exists a subgroup $G_1\leqslant G_0$ that is proper and non-trivial, and $B\cdot B$ is the union of cosets of $G_1$ and a proper subset of one more coset (note that $B\cdot B$ has trivial stabiliser so it cannot be the union of cosets of $G_1$).

We show that the latter case cannot occur. Assume it does. Since $B$ and $B\cdot B$ are complement sets, there is exactly one coset of $G_1$ that contains elements of both $B$ and $B\cdot B$. Let this coset be $g\cdot G_1$.

In this paragraph we show that $g\cdot G_1=G_1$. 
Take any $x\in (B\cdot B)\cap(g\cdot G_1)$. As $x\in B\cdot B$, there exist $y,z\in B$, such that $x=yz$. Assume $y\not\in g\cdot G_1$, so $y\cdot G_1$ is a coset that is entirely contained inside $B$. Therefore $g\cdot G_1=x\cdot G_1= (y\cdot G_1) \cdot z\subset B\cdot B$, which is a contradiction as $g\cdot G_1$ contains elements of $B$, but $B$ and $B\cup B$ are disjoint. So $y\in  g\cdot G_1$ and by symmetry $z\in g\cdot G_1$. Therefore the equation $x=yz$ implies at the level of cosets  $g\cdot G_1=(g\cdot G_1)\cdot (g\cdot G_1)$, so $g\cdot G_1=G_1$ indeed.

Let $u\in G_1\cap B$ and take any $v\in B$ that is not in $G_1$. Such a $v$ exist otherwise $B$ and $B\cdot B$ are both inside $G_1$ so their union cannot be $G$. Thus $u v \in v\cdot G_1\subset B$, but $u v\in B\cdot B$, which contradicts the assumption that $B$ and $B\cdot B$ have empty intersection. 

Therefore $B\cdot B$ is a geometric progression, with ratio $x\in G_0$, say. Each geometric progression of difference $x$ is contained inside a coset of the subgroup generated by $x$. So $B\cdot B$ is contained in one such coset, but $|B\cdot B|>|G_0|/2$, so this coset must be the whole $G_0$. So $x$ generates $G_0$. Therefore $G_0\cong C_{3k+2}=\{1,x,\ldots , x^{3k+1}\}$. Since $B=G_0\backslash (B\cdot B)$, $B$ is a geometric progression too. Let $B=\{a,ax,\ldots ,ax^k\}$, so $B\cdot B=\{a^2,a^2x,\ldots, a^2x^{2k} \}$. But
$B\cdot B=G_0\backslash B=\{ax^{k+1}, ax^{k+2},\ldots ax^{3k+1}\}$, so $a=x^{k+1}$ which proves the lemma. 
\end{proof}

By Lemma \ref{taovu} there is some $x\in G$ such that 
\begin{equation}
    \label{Acontained}
    A\subset x^{k+1}H \cup x^{k+2}H\cup \cdots \cup x^{2k+1}H. 
\end{equation}
Our goal now is to obtain a contradiction from this using the real part trick, similar to the proof of Theorem \ref{t2}. Let $f:\mathbb{R}_{\geq 0}\rightarrow \mathbb{R}$ defined as 
$$f(t)=
\begin{cases}
6/5-t & \text{ if } 0\leq t\leq 6/5, \\ 
0 & \text{ if }  t>6/5. \\
\end{cases}$$

By the first estimate in Proposition \ref{propheath} we have
\begin{equation}
\label{trivial}
S(\chi_0):=\frac{50}{36\log q}\sum_{p}\frac{\chi_0(p) \log p}{p}f\bigg( \frac{\log p}{\log q}\bigg)= 1+o(1),
\end{equation}
and if $\chi\neq \chi_0$ mod $q$ is of bounded order then by the second estimate in Proposition \ref{propheath} we have
\begin{equation}
\label{non-trivial}
S(\chi):=\frac{50}{36\log q}\Re \sum_{p}\frac{\chi(p)\log p}{p}f\bigg( \frac{\log p}{\log q}\bigg)\leq \frac{5}{24} +o(1).
\end{equation}
 Our goal is now to find a linear combination of characters mod $q$ lifted from characters on $G/H$, say $\sum_{i\in I} \alpha_i\chi_i$ with $\alpha_i\in \mathbb{R}$, such that using \eqref{trivial} and \eqref{non-trivial} we get $\sum_{i\in I}\alpha_i S(\chi_i)<0$, however using \eqref{Acontained} we get $\sum_{i\in I}\alpha_i S(\chi_i)> 0$, which is clearly a contradiction. As $k\leq 9$, any non-trivial character lifted from $G/H$ has bounded order, so \eqref{non-trivial} does hold for such characters.

Let $\chi_1$ be the character mod $q$ for which $\chi_1(H)=\{1\}$ and $\chi_1(x)=e\big(\frac{1}{3k+2}\big)$. Let us take
\begin{equation*}
    M:=2.7 S(\chi_1^2)+1.8S(\chi_1^3)+0.29S(\chi_1^6)-S(\chi_0).
\end{equation*}
As $2.7+1.8+0.29<4.8=\frac{24}{5}$, if $q$ is large, then by \eqref{trivial} and \eqref{non-trivial} we have $M<0$ indeed.

We will now show that $M> 0$ using \eqref{Acontained}. We have
\begin{equation*}
    M=\frac{50}{36\log q}\sum_{p} \frac{\log p}{p} f\bigg( \frac{\log p}{\log q}\bigg) \Re\big((2.7\chi_1^2(p)+1.8\chi_1^3(p)+0.29\chi_1^6(p)-\chi_0(p)\big),
\end{equation*}
so as $f$ is non-negative and supported on $[0,6/5]$, it suffices to show that for every 
$p\leq q^{6/5}$, $(p,q)=1$ we have
\begin{equation}
\label{linearcombination}
\Re\big((2.7\chi_1^2(p)+1.8\chi_1^3(p)+0.29\chi_1^6(p)-\chi_0(p)\big)>0.
\end{equation}
By \eqref{Acontained}, for any $p\leq q^{6/5}$, $(p,q)=1$ there is some $k\leq 9$ and $k+1\leq l\leq 2k+1$ such that $p\in x^l H$, which means that $\chi_1(p)=e\big(\frac{l}{3k+2}\big)=e(z)$ for some $z\in \big[\frac{k+1}{3k+2}, \frac{2k+1}{3k+2} \big]\subset \big[ \frac{10}{29}, \frac{19}{29}\big]$. So the left hand side of \eqref{linearcombination} becomes
\begin{equation}
\label{cosinefunction}
    2.7\cos(2\cdot 2\pi z)+1.8\cos(3\cdot 2\pi z)+0.29\cos(6 \cdot 2\pi z)-1.
\end{equation}
This function is symmetric about $0.5$, so we consider its behaviour for $z\geq 0.5$. At $z=0.5$ its value is $0.19$ and as $z$ grows it decreases initially and reaches its local minimum at $z=0.564\ldots $, with value $0.014\ldots $, then starts to grow and reaches its local maximum at $z=0.627\ldots $, with value $0.273\ldots$, after which it starts to decrease and reaches 0 at $z=0.656\ldots >\frac{19}{29}$. Thus we see that \eqref{cosinefunction} is a positive quantity when $z\in  \big[ \frac{10}{29}, \frac{19}{29}\big]$, so \eqref{linearcombination} is positive indeed.
\section{Acknowledgements}
The author was funded through the Engineering and Physical Sciences Research Council Doctoral Training Partnership at the University of Warwick. I would like to thank my supervisor Adam Harper for many helpful discussions and suggestions. I am also greatly indebted to Olivier Ramaré for sending me his latest work on the topic, which made me aware of the work of Iwaniec and Mikawa on the existence of primes in arithmetic progressions. I would like to thank Aled Walker for sending me his PhD thesis, which contains his improved results that I had been previously unaware of.
\printbibliography
\end{document}